\documentclass[10pt]{amsart}

\title[Measures of maximal entropy for suspension semi-flows]{Measures of maximal entropy on subsystems of topological suspension semi-flows}
\author{Tamara Kucherenko}\address{Department of Mathematics, The City College of New York, New York, NY, 10031}\email{tkucherenko@ccny.cuny.edu}

\author{Daniel J. Thompson}\address{Department of Mathematics, Ohio State University, Columbus, OH, 43210}\email{thompson.2455@osu.edu}

\date{\today}

\thanks{T.K. is supported by grants from the Simons Foundation \#430032 and from the PSC-CUNY TRADA-47-18. D.T. is supported by NSF grants DMS-$1461163$ and DMS-$1954463$.}


\newtheorem{thm}{Theorem}[section]

\newtheorem{prop}[thm]{Proposition}
\newtheorem{cor}[thm]{Corollary}

\theoremstyle{definition}

\numberwithin{equation}{section}

\def\Pb{\ifmmode{\Bbb P}\else{$\Bbb P$}\fi}
\def\Z{\ifmmode{\Bbb Z}\else{$\Bbb Z$}\fi}
\def\Q{\ifmmode{\Bbb Q}\else{$\Bbb Q$}\fi}
\def\C{\ifmmode{\Bbb C}\else{$\Bbb C$}\fi}
\def\R{\ifmmode{\Bbb R}\else{$\Bbb R$}\fi}
\def\N{\ifmmode{\Bbb N}\else{$\Bbb N$}\fi}
\def\H{\ifmmode{\Bbb H}\else{$\Bbb H\Bbb N$}\fi}

\def\M{\mathcal M}

\usepackage{graphicx}
\usepackage{amsmath,amssymb,amsthm,amsfonts,amscd,flafter,epsf,bm}
\usepackage{mathtools}
\usepackage{graphics}
\usepackage{MnSymbol}
\usepackage{epsfig}
\usepackage{psfrag}
\usepackage{color}
\input xy
\xyoption{all}

\begin{document}

\begin{abstract}
Given a compact topological dynamical system $(X, f)$ with positive entropy and upper semi-continuous entropy map, and any closed invariant subset $Y \subset X$ with positive entropy, we show that there exists a  continuous roof function such that the set of measures of maximal entropy for the suspension semi-flow over $(X,f)$ consists precisely of the lifts of measures which maximize entropy on $Y$. This result has a number of implications for the possible size of the set of measures of maximal entropy for topological suspension flows. In particular, for a suspension flow on the full shift on a finite alphabet, the set of ergodic measures of maximal entropy may be countable, uncountable, or have any finite cardinality.
\end{abstract}

\maketitle

\setcounter{secnumdepth}{2}

\section{Introduction}

We investigate which measures can achieve maximal entropy in the class of topological suspension semi-flows. Our result is a kind of universality result for measures of maximal entropy (MME) for suspension semi-flows.  We show in Theorem \ref{main} that for \emph{any} positive entropy compact topological dynamical system $(X, f)$ with upper semi-continuous entropy map, and \emph{any} closed invariant subset $Y \subset X$ with positive entropy, we can find a suspension semi-flow over $(X, f)$ whose MME  are exactly the lifts of the MME 
for the subsystem $(Y, f|_{Y})$.  The assumption that the entropy map is upper semi-continuous is needed for existence of the MME and is thought of as a very weak expansivity condition.  The condition that the \mbox{subsystem} has positive entropy is essential since any suspension flow over $X$ has positive entropy, so only positive entropy measures in the base can possibly lift to MME for the flow.

Universality and structure results for suspension flows are well known in the measure-theoretic category. For example, Rudolph \cite{dR76} proved the famous result that every measure-preserving flow is isomorphic to a suspension flow whose roof takes only two values.  Quas and Soo \cite{QS} proved the following measure-theoretic universality result for H\"older suspension flows over the full shift: any measure of smaller entropy than the flow can be embedded isomorphically into the suspension flow. The main result of our paper can be interpreted as a universality phenomenon in the topological category.

After our main result is proved, it can be used as a machine for producing examples of suspension flows with various different prescribed behaviors through careful choices for the subsystem $Y$.
When the base is the full shift $(\Sigma, \sigma)$, by choosing  $Y \subset \Sigma$ appropriately, we can build topological suspension flows over the shift $(\Sigma, \sigma)$ with any of the following properties: the MME is unique but not fully supported; there are multiple MMEs with the same support; the MMEs (unique or not) are supported on a minimal subsystem; there is any prescribed finite number of ergodic MMEs; the set of ergodic MMEs is countably infinite; the set of ergodic MMEs is uncountable.
These statements contrast sharply with the well-known situation of a suspension flow over the full shift on a finite alphabet with H\"older continuous roof function in which case the MME  is unique and fully supported. Note that every continuous suspension flow over a full shift with a fixed alphabet size is orbit equivalent. Thus, our results show that a continuous orbit equivalence does not preserve finiteness, or even countability,  of the set of MME.

We discuss previous results in this direction and our approach. We extensively generalize our previous work \cite{KTh} in which we proved that for the full shift $(\Sigma, \sigma)$ and any positive entropy subshift of finite type $Y \subset \Sigma$, there exists a roof function such that the MMEs for the suspension flow over the full shift are exactly the lifts of the MMEs for the subshift. In the current paper, we remove the restriction that the subset $Y$ is a shift of finite type, and we remove the need for the ambient space to be symbolic, allowing any topological dynamical system with upper semi-continuous entropy map.

The proof in \cite{KTh} was based on an explicit combinatorial description of the roof function, and hands-on pressure estimates. That argument has the advantage of giving explicit and constructive examples for which uniqueness of the MME fails in the class of topological suspension flows. However, that argument relied heavily on the structure of subshifts of finite type, and did not provide examples with multiple MME on a single transitive component. Our explicit construction does not seem to carry over even to sofic subshifts of the full shift, much less to general subshifts or topological dynamical systems.

We note that Iommi and Velozo \cite{IV} have recently given an independent proof that suspension flows over the full shift on a finite alphabet can have an uncountable set of ergodic MME. They show that this phenomenon is dense in the sense that for any continuous roof function over a compact shift of finite type, there is a small continuous time change so that the resulting flow has an uncountable collection of ergodic MME.

The approach of this paper is a generalization and refinement of a result by Markley and Paul \cite{MP82}, and is based on the theory of tangent functionals. The MMEs for the flow are described in terms of equilibrium states of the base system (see \S\ref{s.flows}).  We show that for systems with upper semi-continuous entropy map, we can find a function $\tau$ whose equilibrium states are the MME on a given subsystem, and in addition that the function is non-positive and vanishes only on the subsystem. Markley and Paul proved a version of this result under the hypothesis that the ambient space is the full shift, and without obtaining the key conclusion needed for the application to suspension flows, which is that the function is non-positive. The desired roof function, which must be strictly positive, is obtained by adding a suitable normalization constant to $-\tau$.

The paper is structured as follows. In \S \ref{s.prelim}, we collect our preliminaries. In \S \ref{s.mainresult}, we state and prove our main result. In \S \ref{s.applications}, we apply our main result to questions about the set of possible MMEs for suspensions over the full shift.

\section{Preliminaries} \label{s.prelim}
\subsection{The pressure function and tangent functionals} \label{tangentfunctionals}
We consider a continuous function $f:X\to X$ on a compact metric space $X$. We call the pair $(X, f)$ a topological dynamical system. We denote by $C(X)$ the Banach
space of all continuous real-valued functions of $X$ with the supremum norm. The dual space $C^\ast (X)$ consists of all Radon measures on $X$. Let $\M(f) \subset C^\ast(X)$ be the subspace of $f$-invariant probability measures.

For $\phi\in C(X)$, the \emph{topological pressure} can be defined by the Variational Principle to be the quantity
\begin{equation}\label{VarPr}
P(\phi)=\sup_{\mu\in\M(f)}\left\{h_\mu(f)+\int \phi\, d\mu\right\},
\end{equation}
where $h_\mu(f)$ is the measure-theoretic entropy of $\mu$. The topological pressure can also be defined using $(n, \epsilon)$-separated or spanning sets, see \cite{pW82} for a detailed treatment. The number $P(0)=\sup_{\mu\in\M(f)}\{h_\mu(f)\}$ is called the \emph{topological entropy} of $f$ and is denoted by $h_{\rm top}(f)$. If there exists a measure $\mu\in\M(f)$ at which the supremum in (\ref{VarPr}) is attained it is called an \emph{equilibrium state} of $\phi$. The set of all equilibrium states of $\phi$ is a convex subset of $\M(f)$ which is compact with respect to the weak$^\ast$-topology. However, in general it may be empty \cite{Gu}. If the entropy map $\mu\mapsto h_\mu(f)$ is upper semi-continuous on $\M(f)$ then for any $\phi\in C(X)$ the topological pressure $P(\phi)< \infty$ and  the set of equilibrium states is not empty. This condition will be a hypothesis of our main result.

The \emph{pressure function} (with respect to $f$) is the map $P:C(X)\to \R \cup\{\infty\}$ which sends $\phi$ to $P(\phi)$. The pressure function is continuous, convex and satisfies the following properties \cite{pW82}:
\begin{enumerate}
  \item $P$ is Lipschitz, i.e. $|P(\phi)-P(\psi)|\le\|\phi-\psi\|$ for any $\phi,\psi\in C(X)$;
  \item $P$ is increasing, i.e. $P(\phi)\le P(\psi)$ whenever $\phi\le \psi$;
  \item $P(t+\phi+\psi\circ f-\psi)=t+P(\phi)$ for any $t\in\R$ and $\phi,\psi\in C(X)$.
\end{enumerate}

If $f$ has finite topological entropy, then the function $P$ is real-valued. This is the case for us since we assume upper semi-continuity of the entropy map. If $f$ has infinite topological entropy, then $P(\phi)=\infty$ for all $\phi\in C(X)$. We will make use of the following facts from the theory of tangent functionals for convex functions. For a continuous convex function $Q:C(X)\to \R$, a measure $\nu\in C^\ast(X)$ is called \emph{$Q$-bounded} if there is $C\in\R$ such that for any $\psi\in C(X)$ we have
\begin{equation}\label{Q-bnd}
  \int\psi\, d\nu\le Q(\psi)+C.
\end{equation}
We say that $\nu\in C^\ast(X)$ is a \emph{tangent} to $Q$ at $\phi\in C(X)$ if for any $\psi\in C(X)$ we have
\begin{equation}\label{Tangent_def}
 \int\psi\, d\nu\le Q(\phi+\psi)-Q(\phi).
\end{equation}
A simple application of the Hahn-Banach theorem shows that the set of tangents is non-empty at any $\phi\in C(X)$. We can approximate any $Q$-bounded measures by tangents using the following result which is a special case of Israel's theorem \cite{I}. For purposes of exposition and to make the paper self-contained, we include a short proof.
\begin{prop} [Special case of Israel's Theorem] \label{Israel}
Let $\varepsilon>0$, $\mu\in C^\ast(X)$ be a $Q$-bounded measure, and $V\subset C(X)$ be a closed linear subspace. Then there exists a function $\phi\in V$ and a tangent $\nu$ to $Q$ at $\phi$ such that $\|\mu-\nu\|_{C^\ast(V)}<\varepsilon$.
\end{prop}
\begin{proof}
Since $\mu$ is $Q$-bounded, there is $C\in \R$ such that $\int\xi d\mu\le Q(\xi)+C$ for all $\xi\in C(X)$. We may assume that both $\mu=0$ and $C=0$ by replacing the function $Q(\xi)$ with $Q(\xi)-\int\xi d\mu-C$. For $\psi\in V$, we define
\begin{equation}\label{S(omega)_def}
  S(\psi)=\{\xi\in V:Q(\psi)-Q(\xi)\ge \varepsilon\|\psi-\xi\|\}
\end{equation}
Since $\psi\in S(\psi)$ and $Q$ is continuous, $S(\psi)$ is a nonempty closed subset of $V$. Moreover, it is easy to see that $S(\xi)\subset S(\psi)$ whenever $\xi\in S(\psi)$.

We pick any starting point $\psi_0\in V$ and build a sequence $(\psi_n)_{n\ge 1}$ such that $\psi_n\in S(\psi_{n-1})$ and
\begin{equation}\label{psi_n_def}
 Q(\psi_n)<\inf_{\xi\in S(\psi_{n-1})}Q(\xi)+\frac{\varepsilon}{2^n}.
\end{equation}
Note that $S(\psi_n)$ is a sequence of nested closed sets whose diameters tend to zero. Indeed, for $\xi\in S(\psi_n)$ we have
\begin{equation}\label{Size_S(psi_n)}
  \varepsilon\|\psi_n-\xi\|\le Q(\psi_n)-Q(\xi)\le \inf_{ S(\psi_{n-1})}Q+\frac{\varepsilon}{2^n} -Q(\xi)\le \frac{\varepsilon}{2^n},
\end{equation}
where the last inequality follows from the fact that $\xi$\ is also in $S(\psi_{n-1})$. In particular,  we obtain that $(\psi_n)$ is Cauchy and hence there exists $\phi=\lim \psi_n$. Furthermore, (\ref{Size_S(psi_n)}) implies that $S(\phi)$ contains only one point $\phi$, since $S(\phi)\subset S(\psi_{n})$ for all $n$.

Now we define two subsets in $C(X)\times \R$ by
\[
  A=\{(\xi,y):\xi\in V,\,\, y<Q(\phi)-\varepsilon\|\xi-\phi\|\}\quad\text{and}\quad B=\{(\xi,y): y>Q(\xi)\}.
\]
It follows from the continuity and convexity of $Q$ that $A$ and $B$ are open and convex. If $(\xi,y)\in A\cap B$ then $\xi\in S(\phi)$ and we must have $\xi=\phi$. Since $(\phi,y)$ cannot lie in both $A$ and $B$, we conclude that $A\cap B$ is empty. By the Hahn-Banach theorem there exists a linear functional $\lambda$ on  $C(X)\times \R$ and $t\in\R$ such that $\lambda(a)<t<\lambda(b)$ for all $a\in A$ and $b\in B$. With a suitable choice of $t\in\R$ and $\nu\in C^*(X)$, after a possible rescaling of $\lambda$ by a positive constant we can write $\lambda(\xi,y)=y-\int\xi d\nu$. Then for any $\xi\in V$ we have $y-\int\xi d\nu< t$ as long as $y<Q(\phi)-\varepsilon\|\xi-\phi\|$. This gives
\begin{equation}\label{norm_estimate}
 Q(\phi)-\varepsilon\|\xi-\phi\|-\int\xi d\nu\le t.
\end{equation}
On the other hand, for any $\xi\in C(X)$ we have $t<y-\int\xi d\nu$ as long as $y>Q(\xi)$. In that case we arrive at
\begin{equation}\label{tangent_estimate}
t\le Q(\xi)-\int\xi d\nu.
\end{equation}
Taking $\xi=\phi$ in both inequalities above we see that $t=Q(\phi)-\int\phi d\nu$. With $\psi=\xi-\phi$ this allows us to rewrite (\ref{norm_estimate}) and (\ref{tangent_estimate}) as
\begin{align*}
  &\int \psi d\nu\le\varepsilon\|\psi\|,\,\text{for any}\,\psi\in V;  \\
  &\int\psi d\nu\le Q(\psi+\phi)-Q(\phi),\,\text{for any}\,\psi \in C(X).
\end{align*}
The first inequality provides the required norm estimate $\|\nu\|_{C^*(V)}\le\varepsilon$ and the second one shows that $\nu$ is a tangent to $Q$ at $\phi$.
\end{proof}

Since the pressure function $P$ is continuous and convex, there exists a tangent to the pressure at every $\phi\in C(X)$. From properties (1) and (3) of the pressure  one can deduce that any tangent to $P$ must be a positive $f$-invariant probability measure on $X$. On the other hand, if $\mu\in\M (f)$ is an equilibrium state of $\phi$, then by the Variational Principle for any $\psi\in C(X)$ we have
$$P(\phi+\psi)-P(\phi)\ge h_\mu(f)+\int(\phi+\psi)d\mu-h_\mu(f)-\int\phi d\mu=\int\psi d\mu,$$
and hence $\mu$ is tangent to $P$ at $\phi$. It follows that the set of equilibrium states for $\phi$ is a subset of tangents to $P$ at $\phi$. It was proved in \cite[Theorem 9.15]{pW82} that the opposite inclusion holds if the entropy map $\mu\mapsto h_\mu(f)$ is upper semi-continuous on the set of tangents to $P$ at $\phi$. In \cite[Theorem 5]{pW92}, it was shown that this statement is an `if and only if'.

\subsection{Suspension Semiflows}\label{s.flows}
We recall some basic facts about suspension semiflows. Given a topological dynamical system $(X, f)$ and a continuous function $\rho:X\to (0,\infty)$ consider the quotient space
\begin{equation}\label{SuspSpace}
  ^{\textstyle X_\rho=\{(x,s): x \in X,  0\le s\le \rho(x) \}}\big/_{\textstyle \sim}
\end{equation}
obtained by the equivalence relation $(x,\rho(x))\sim (f(x),0)$ for every $x\in X$. We refer to $X$ as the \emph{base}, to $\rho$ as the \emph{roof function} and to $X_\rho$ as the \emph{suspension space relative to $\rho$}. The \emph{suspension semiflow} $\Phi = (\varphi_t)_{t\ge 0}$ on $X_\rho$ associated to $(f,X,\rho)$ is defined locally by $\varphi_t(x,s)=(x, s+t)$. We extend this definition to all $t\in [0, \infty)$ by setting
\begin{equation} \label{semiflow}
\varphi_t(x,s)=\left(f^n(x), t+s-\sum_{k=0}^{n-1} \rho(f^k(x))\right),
\end{equation}
where $n\in\N$ is uniquely determined by  $\sum_{k=0}^{n-1} \rho(f^k(x))\le t+s< \sum_{k=0}^{n} \rho(f^k(x))$. We call a flow of this type a \emph{topological suspension semiflow} to emphasize that we are working with topological dynamical systems and continuous roof functions.

Since $\rho$ is bounded away from zero there is a natural identification between the space $\M(\Phi)$ of $\Phi$-invariant
probability measures on $X_\rho$ and the space $\M(f)$ of $f$-invariant probability measures on $X$. If $m$ denotes the Lebesgue measure on $\R$ then the map
\begin{equation}\label{lifting of measures}
 \mu \mapsto \tilde \mu =\frac{(\mu\times m)|_{X_\rho}}{\int_X \rho\, d\mu}
\end{equation}
is a bijection from $\M(f)$ to $\M(\Phi)$. We say that $\tilde \mu$ is the \emph{lift} of the measure $\mu$. Abramov \cite{lA59} established a relation between the entropies of the measure $\mu$ and the lifted measure $\tilde\mu$, namely
\begin{equation}\label{Abramov}
h_{\tilde \mu}(\varphi_t)=\frac{t\cdot h_\mu(f)}{\int_X\rho\, d\mu}.
\end{equation}
Therefore, the entropy $h_{\tilde\mu}(\Phi)$ of the measure $\tilde\mu$ with respect to the flow, which is defined as the entropy of the time-one map $\varphi_1$ satisfies
\begin{equation}\label{Abramov2}
h_{\tilde\mu}(\Phi)=\frac{h_\mu(f)}{\int_X\rho\, d\mu}.
\end{equation}

We remark that if $f:X\to X$ is a homeomorphism then the expression (\ref{semiflow}) with $n\in\Z$ defines a flow $(\varphi_t)_{t\in\R}$ on $X_\rho$. In this case the bijection between the measure spaces (\ref{lifting of measures}) remains the same and Abramov's formula (\ref{Abramov}) is valid with $t$ replaced by $|t|$.

The measures of maximal entropy for the semiflow $(X_\rho,\Phi)$ can be described in terms of equilibrium states of a constant multiple of the roof function on the base space \cite{BR, PP90}. Precisely, for $c\in \R$ consider the function $c \to P(-c \rho)$. It follows from (\ref{VarPr}) that this function is real-valued and strictly decreasing, and hence there exists $c$ with $P(-c \rho)=0$. Suppose $-c\rho$ has an equilibrium state $\mu$. Denote by $\tilde{\mu}$ the image of $\mu$ given by (\ref{lifting of measures}). We claim that $\tilde{\mu}$ is a measure of maximal entropy for the flow $\Phi$. Indeed, let $\tilde{\nu}$ be any other $\Phi$-invariant measure on $X_\rho$ and $\nu$ be the corresponding $f$-invariant measure on the base. By the Variational Principle (\ref{VarPr}) we have
\[
0=h_\mu + \int-c\rho d\mu \geq h_\nu + \int-c\rho d\nu
\]
 with equality if and only if $\nu$ is an equilibrium state for $-c\rho$. It follows from (\ref{Abramov2}) that
\[
h_{\tilde{\mu}}(\Phi)=\frac{h_\mu(f)}{\int \rho d \mu} \geq \frac{h_{\nu}(f)}{\int \rho d \nu}=h_{\tilde{\nu}}(\Phi)
\]
and $\tilde{\mu}$ is a measure of maximal entropy for the flow. Conversely, any measure of maximal entropy for $(X_\rho,\Phi)$ corresponds to an equilibrium state of $-c\rho$ on the base transformation $(X, f)$ with $c=h_{\rm top}(\Phi)$.

\section{Main Result} \label{s.mainresult}
We state and prove our main theorem.
\begin{thm}\label{main}
Let $(X, f)$ be a compact topological dynamical system such that $h_{\rm top}(f)>0$ and the entropy map $\mu\mapsto h_\mu(f)$ is upper semi-continuous. Let  $Y\subset X$ be a closed $f$-invariant subset with $h_{\rm top}(f|_Y)>0$.  There exists a continuous roof function $\rho:X\to (0,\infty)$ so that the measures of maximal entropy for the suspension semi-flow $(X_{\rho},\Phi)$ (or suspension flow if $f$ is a homeomorphism) are exactly the lifts of the measures of maximal entropy for $Y$.
\end{thm}
The assumption $h_{\rm top}(f|_Y)>0$ is essential. If $Y$ satisfies $h_{\rm top}(f|_Y)=0$, then by Abramov's formula, the lift of a measure of maximal entropy on $
Y$ has entropy $0$.  Since $h_{\rm top}(\Phi)>0$ for any continuous roof function $\rho$, such a measure can never be an MME for any topological suspension semi-flow over $(X,f)$.
\begin{proof} We denote by $V$ the set of all continuous functions which vanish on $Y$, i.e.
 \begin{equation}\label{Def_V}
 V=\{\psi\in C(X): \psi(x)=0\text{ whenever }x\in Y\}.
\end{equation}
Then $V$ is a closed linear subspace of $C(X)$. Consider $\mu\in\M (f)$ such that $\mu(Y)=1$. The Variational Principle (\ref{VarPr}) implies that any $f$-invariant probability measure is $P$-bounded. Hence, we can apply Proposition \ref{Israel} to the subspace $V$ and measure $\mu$ with $\varepsilon=1/2$.
We obtain the existence of $\phi\in V$ and a tangent $\nu$ to $P$ at $\phi$ such that for any $\psi\in V$ we have
\begin{equation}\label{IsraelThm}
 \left|\int \psi\,d\nu-\int\psi\,d\mu\right|\le\frac12\|\psi\|.
\end{equation}
From our assumption that the entropy map $\mu\mapsto h_\mu(f)$ is upper semi-continuous, and the discussion in \S \ref{tangentfunctionals}, it follows that $\nu$ is an equilibrium state for $\phi$. Note that since $\mu$ is supported on $Y$ and $\psi$ vanishes on $Y$, the second integral in (\ref{IsraelThm}) is zero, and thus $| \int \psi\,d \nu| \le \frac12\|\psi\|$.

First we show that $Y$ has positive $\nu$-measure. Assume to the contrary that $\nu(Y)=0$. Then the complement of $Y$ is open and has full $\nu$-measure. Since $\nu$ is regular, we can approximate the complement of $Y$ by closed sets from below. Let $F\subset X\setminus Y$ be a closed set such that $\nu(F)>1/2$. The sets $Y$ and $F$ are closed and disjoint, thus by Urysohn's Lemma they can be separated via continuous functions. Precisely, there exists a continuous $\psi_0:X\to [0,1]$ such that $\psi_0(x)=0$ for any $x\in Y$ and $\psi_0(x)=1$ for any $x\in F$. We obtain that $\psi_0\in V$ and
\begin{equation*}
  \int\psi_0\,d\nu\ge\int_{F}\psi_0 d\nu=\nu(F)>\frac12,
\end{equation*}
which contradicts (\ref{IsraelThm}).

The set of all equilibrium states for $\phi$ is a compact and convex subset of $\M(f)$  whose
extreme points are the ergodic measures. Therefore, the ergodic decomposition of $\nu$ must contain at least one ergodic equilibrium state $\nu_E$ with $\nu_E(Y)>0$. The ergodicity of $\nu_E$ and $f$-invariance of $Y$ imply that, in fact, $\nu_E(Y)=1$. Since $\phi|_Y\equiv 0$ we obtain
\begin{equation*}
  P(\phi)=h_{\nu_E}(f)+\int \phi\,d\nu_E=h_{\nu_E}(f).
\end{equation*}
Moreover, for any other invariant measure $m$ supported on $Y$ we have $h_{m}(f)=h_m(f)+\int \phi\,d{m}\le P(\phi)=h_{\nu_E}(f)$. We conclude that $P(\phi)=h_{\rm top}(f|_Y)$ and hence every measure of maximal entropy of $f|_Y$ is an equilibrium state of $\phi$.

However, a priori the function $\phi$ might have some other equilibrium states which are not supported on $Y$.
To eliminate this possibility, we make use of a continuous function which satisfies $\xi(x)=0$ for $x\in Y$ and $\xi(x)>0$ for $x \in X \setminus Y$. We define $\tau=\min\{0,\phi\}-\xi$. Then $\tau$ is continuous, $\tau(x)=0$ whenever $x\in Y$ and $\tau(x)<0$  whenever $x\notin Y$. In addition, we have that $\tau\le\phi$, which implies that $P(\tau)\le P(\phi)$.
For $m \in \M(f)$ which is an MME for $f|_Y$, we have $P(\tau) \geq h_m(f)+\int \tau\,d{m} = h_{m}(f) =h_{\rm top}(f|_Y) = P(\phi)$. We conclude that $P(\tau)= P(\phi)$. Thus, $P(\tau)=h_{\rm top}(f|_Y)$ and any MME for $f|_{Y}$ is an equilibrium state for $\tau$. A measure supported on $Y$ which is not an MME for $f|_{Y}$ is clearly not an equilibrium state for $\tau$. For a measure $\lambda \in  \M(f)$ with $\lambda(X\setminus Y )>0$, note that $\int \xi d \lambda >0$, and thus we have
\begin{align*}
  h_{\lambda}(f)+\int \tau\,d{\lambda} & \leq h_{\lambda}(f)+\int \phi d \lambda - \int \xi d \lambda \\
   & < h_{\lambda}(f)+\int \phi\,d{\lambda}\\
   &\le P(\phi).
\end{align*}
Hence $\lambda$ is not an equilibrium state for $\tau$. We conclude that the set of equilibrium states of $\tau$ is exactly the set of measures of maximal entropy for $Y$.
We have constructed a continuous $\tau:X\to\R$ satisfying
\begin{enumerate}
  \item $\tau(x)=0$ for any $x\in Y$ and $\tau(x)<0$ for any $x\notin Y$;
  \item $P(\tau)=h_{\rm top}(f|_Y)$;
  \item the set of equilibrium states for $\tau$ is exactly the set of measures of maximal entropy for $Y$.
\end{enumerate}
We define $\rho: X \to (0, \infty)$ by $\rho=P(\tau) - \tau$. The function $\rho$ is positive because $\tau-P(\tau) = \tau - h_{\rm top}(f|_Y)<0$.  The function $-\rho$ has the same set of equilibrium states as $\tau$, and we have $P(-\rho) = P(\tau-P(\tau)) =0$. It follows from the discussion in \S \ref{s.flows} that the measures of maximal entropy for the suspension semi-flow with roof function $\rho$ are exactly the lifts of the measures of maximal entropy for $Y$.
\end{proof}

\section{ MMEs for suspensions over the full shift} \label{s.applications}
We now apply our main result, mainly in the case of suspension flows over the full shift on a finite alphabet.

\subsection{Support for unique MME}
Theorem \ref{main} allows for much flexibility in specifying the support of a unique measure of maximal entropy. Recall that a system is \emph{intrinsically ergodic} if it has a unique MME.
\begin{cor} \label{cor1}
Let $(X, f)$ be a compact topological dynamical system such that $h_{\rm top}(f)>0$ and the entropy map $\mu\mapsto h_\mu(f)$ is upper semi-continuous.  Let $Y$ be a closed $f$-invariant subset so that $h_{\rm top}(f|_Y)>0$ and $f|_Y$ is intrinsically ergodic. There exists a continuous $\rho:X\to (0,\infty)$ so that the suspension semi-flow $(X_\rho,\Phi)$ is intrinsically ergodic and its unique measure of maximal entropy is supported on the lift of $Y$ to $X_\rho$.
\end{cor}
This follows immediately from Theorem \ref{main}. Examples of suitable $Y \subset X$ include
expansive subsystems with specification \cite{rB74}. For example, $Y$ could be a transitive horseshoe inside a smooth non-uniformly hyperbolic system $(X,f)$. Corollary \ref{cor1} can be applied to the full shift and a positive entropy uniquely ergodic subsystem to give an example of a unique MME supported on a minimal set. Explicit examples of such subsystems were provided by Grillenberger \cite{cG73}.  Another example where the structure of $Y$ is more complex is to let $(X,f)$ be the full shift on $\{0, 1, \ldots, n\}$, and $Y=\Sigma_{\beta}$ be the $\beta$-shift for some $\beta \in (1, n)$. The $\beta$-shift is intrinsically ergodic, but typically does not have the specification property.

\subsection{Two ergodic MMEs on a transitive component} We now turn to the phenomenon of non-uniqueness of MME.  Examples of suspension flows over the full shift with non-unique MME were given in \cite{KTh}, however each MME was supported on a different transitive component. Theorem \ref{main} gives us the following corollary.
\begin{cor}There are examples of suspension
flows over the full shift on a finite alphabet which have
two distinct ergodic measures of maximal entropy with the same support.
\end{cor}
This can be achieved by letting $X$ be the full shift on four symbols, and taking the subshift $Y$ to be the Dyck shift. We recall the simple description of this shift, given in terms of parentheses and brackets. We split the alphabet of $X$ into two pairs of matching left and right symbols and denote them by (, ), [, ]. The Dyck shift consists of all sequences where every
opening parenthesis ( must be closed by ) and every opening bracket [ must be closed by ].
 Krieger \cite{Kr} showed that the Dyck shift has topological entropy $\log 3$ and admits exactly two ergodic measures of maximal entropy, both are fully supported and Bernoulli. An application of Theorem \ref{main} gives a suspension flow over the full shift on four symbols with two measures of maximal entropy which have the same support. Furthermore, both of these measures are isomorphic to the product of a Bernoulli flow and a rotational flow (see Ledrappier, Lima, Sarig \cite[Lemma 4.8]{LLS}).

There exist minimal subshifts with positive entropy which are not uniquely ergodic. These examples are commented on in, for example, \cite[p.157]{DGS}, and \cite{kP86}. Explicit examples can be constructed by modifying Grillenberger's arguments for uniquely ergodic positive entropy subshifts. These are more difficult to construct than the Dyck example. However, they allow us to strengthen the conclusion of the Corollary from `the same support' to `both supported on the same minimal set'.

\subsection{Cardinality for the set of MME}
In \cite{Kr}, Krieger  extends the Dyck shift construction described above to get examples of transitive subshifts on a finite alphabet with positive entropy and multiple Bernoulli measures as maximal measures. This construction shows that there is no upper bound on the number of possible measures of maximal entropy for a suspension
flow over a full shift, and we can ensure that these MME all have the same support and are isomorphic to the product of a Bernoulli flow and a rotational flow.

The first example of a transitive positive entropy subshift with any prescribed finite number of ergodic measures of maximal entropy was given by Shtilman in \cite{Sh}. An elegant construction was given by Haydn \cite{H}, who showed that for any $L\in\mathbb{N}$ there is a topologically mixing subshift on $2L+1$ symbols with positive topological entropy and $L$ distinct ergodic entropy maximizing measures.

It is also possible for a transitive subshift on a finite alphabet to have infinitely many maximizing measures. In \cite[Lemma~8]{jB05}, Buzzi constructed a subshift on three symbols with topological entropy $\log 2$ which supports countably many ergodic measures with entropy $\log 2$. 
A simple explicit example of a transitive shift with uncountably many ergodic MME is also given in Buzzi \cite[Proof of Lemma~17]{jB05}.
Applying Theorem \ref{main} to the full shift and the subshifts described above, we obtain the following result.
\begin{cor}
For suspension flows over the full shift on a finite alphabet, the set
of ergodic measures of maximal entropy can have any finite cardinality, be countably infinite, or be uncountably
infinite.
\end{cor}
We remark that the existence of suspension flows in the above class for which the set of ergodic MMEs is uncountable has been independently obtained in a preprint by Iommi and Velozo \cite{IV}.

\bibliographystyle{plain}
\bibliography{nonuniqueMMEbiblio}

\begin{thebibliography}{10}

\bibitem{lA59}
L.M. Abramov.
\newblock On the entropy of a flow. ({R}ussian).
\newblock {\em Dokl. Akad. Nauk SSSR}, 128:873--875, 1959.

\bibitem{rB74}
Rufus Bowen.
\newblock Some systems with unique equilibrium states.
\newblock {\em Math. Systems Theory}, 8(3):193--202, 1974/75.

\bibitem{BR}
Rufus Bowen and David Ruelle.
\newblock The ergodic theory of {A}xiom {A} flows.
\newblock {\em Invent. Math.}, 29(3):181--202, 1975.

\bibitem{jB05}
J\'{e}r\^{o}me Buzzi.
\newblock Subshifts of quasi-finite type.
\newblock {\em Invent. Math.}, 159(2):369--406, 2005.

\bibitem{DGS}
Manfred Denker, Christian Grillenberger, and Karl Sigmund.
\newblock {\em Ergodic theory on compact spaces}.
\newblock Lecture Notes in Mathematics, Vol. 527. Springer-Verlag, Berlin-New
  York, 1976.

\bibitem{cG73}
Christian Grillenberger.
\newblock Constructions of strictly ergodic systems. {I}. {G}iven entropy.
\newblock {\em Z. Wahrscheinlichkeitstheorie und Verw. Gebiete}, 25:323--334,
  1972/73.

\bibitem{Gu}
B.~M. Gurevi\v{c}.
\newblock Topological entropy of a countable {M}arkov chain.
\newblock {\em Dokl. Akad. Nauk SSSR}, 187:715--718, 1969.

\bibitem{H}
Nicolai T.~A. Haydn.
\newblock Phase transitions in one-dimensional subshifts.
\newblock {\em Discrete Contin. Dyn. Syst.}, 33(5):1965--1973, 2013.

\bibitem{IV}
G.~Iommi and A.~Velozo.
\newblock Measures of maximal entropy for suspension flows.
\newblock {\em Preprint arXiv:1908.07020, To appear in Mathematische
  Zeitschrift}, 2019.

\bibitem{I}
Robert~B. Israel.
\newblock Existence of phase transitions for long-range interactions.
\newblock {\em Comm. Math. Phys.}, 43:59--68, 1975.

\bibitem{Kr}
Wolfgang Krieger.
\newblock On the uniqueness of the equilibrium state.
\newblock {\em Math. Systems Theory}, 8(2):97--104, 1974/75.

\bibitem{KTh}
Tamara Kucherenko and Daniel~J Thompson.
\newblock Measures of maximal entropy for suspension flows over the full shift.
\newblock {\em Mathematische Zeitschrift}, 294:769--781, 2020.

\bibitem{LLS}
Fran\c{c}ois Ledrappier, Yuri Lima, and Omri Sarig.
\newblock Ergodic properties of equilibrium measures for smooth three
  dimensional flows.
\newblock {\em Comment. Math. Helv.}, 91(1):65--106, 2016.

\bibitem{MP82}
Nelson~G. Markley and Michael~E. Paul.
\newblock Equilibrium states of grid functions.
\newblock {\em Trans. Amer. Math. Soc.}, 274(1):169--191, 1982.

\bibitem{PP90}
W.~Parry and M.~Pollicott.
\newblock {\em Zeta functions and the periodic orbit structure of hyperbolic
  dynamics}.
\newblock Number 187-188 in Ast\'erisque. Soc. Math. France, 1990.

\bibitem{kP86}
Karl Petersen.
\newblock Chains, entropy, coding.
\newblock {\em Ergodic Theory Dynam. Systems}, 6(3):415--448, 1986.

\bibitem{QS}
Anthony Quas and Terry Soo.
\newblock Weak mixing suspension flows over shifts of finite type are
  universal.
\newblock {\em J. Mod. Dyn.}, 6(4):427--449, 2012.

\bibitem{dR76}
Daniel Rudolph.
\newblock A two-valued step coding for ergodic flows.
\newblock {\em Math. Z.}, 150(3):201--220, 1976.

\bibitem{Sh}
M.~S. Shtilman.
\newblock The number of invariant measures with maximal entropy for a shift in
  a sequence space.
\newblock {\em Mat. Zametki}, 9:291--302, 1971.

\bibitem{pW82}
P.~Walters.
\newblock {\em An Introduction to Ergodic Theory}, volume~79 of {\em Graduate
  Texts in Mathematics}.
\newblock Springer, New York, 1982.

\bibitem{pW92}
Peter Walters.
\newblock Differentiability properties of the pressure of a continuous
  transformation on a compact metric space.
\newblock {\em J. London Math. Soc. (2)}, 46(3):471--481, 1992.

\end{thebibliography}

\end{document}